\documentclass[12pt,twoside]{amsart}
\usepackage{amsmath,amssymb}    
\usepackage{url}                
\usepackage{graphicx}           
\usepackage{amsthm}
\usepackage{fullpage}
\usepackage{hyperref}
\theoremstyle{definition}
\numberwithin{equation}{section}
\newtheorem{theorem}[equation]{Theorem}
\newtheorem{lemma}[equation]{Lemma}
\newtheorem{proposition}[equation]{Proposition}
\newtheorem{corollary}[equation]{Corollary}

\newtheorem{definition}[equation]{Definition}
\newtheorem{example}[equation]{Example}
\newtheorem{question}[equation]{Question}
\newtheorem{conjecture}[equation]{Conjecture}
\newtheorem{remark}[equation]{Remark}

\newcommand{\OO}{\mathcal{O}}

\newcommand{\Hom}{\text{Hom}}

\begin{document}

\begin{abstract}
We study Poisson traces of the structure algebra A of an affine Poisson variety $X$ defined over a field of characteristic $p$.
According to \href{http://arxiv.org/abs/0908.3868v4}{arXiv:0908.3868}, the dual space $HP_0(A)$ to the space of Poisson traces arises as the space of coinvariants associated to a certain $D$-module $M(X)$ on X.
If $X$ has finitely many symplectic leaves and the ground field has characteristic zero, then $M(X)$ is holonomic, and thus $HP_0(A)$ is finite dimensional. However, in characteristic $p$,
the dimension of $HP_0(A)$ is typically infinite.  Our main results are complete computations of $HP_0(A)$ for sufficiently large $p$ when $X$ is 1) a quasi-homogeneous isolated surface singularity in the three-dimensional space, 2) a quotient singularity $V/G$, for a symplectic vector space $V$ by a finite subgroup $G$ in $Sp(V)$, and 3) a symmetric power of a symplectic vector space or a Kleinian singularity.  In each case, there is a finite nonnegative grading, and we compute explicitly the Hilbert series. The proofs are based on the theory of $D$-modules in positive characteristic. \end{abstract}

\date{}
\newcommand{\h}{{\mathfrak{h}}}
\newcommand{\CC}{{\Bbb C}}

\title{Poisson traces in positive characteristic}

\author{Yongyi Chen, Pavel Etingof, David Jordan, and Michael Zhang}
\maketitle

\section{Introduction}

Let $A$ be a Poisson algebra over a field $F$.
A {\it Poisson trace} on $A$ is a linear functional $T: A\to F$
such $T(\lbrace{a,b\rbrace})=0$ for $a,b\in A$, i.e., a Lie algebra
character of $A$.  Thus, the space ${\rm Tr}(A)$ of Poisson traces 
is dual to the space $HP_0(A)=A/\lbrace{ A,A\rbrace}$, the {\it
zeroth Poisson homology} of $A$. The space $HP_0(A)$ is an interesting and 
important invariant of $A$. 

In characteristic zero, the space $HP_0(A)$ is 
now reasonably well understood in a number of cases. 
For instance, Alev and Lambre showed in \cite{AL} 
that if $A$ is the algebra of functions 
on a quasihomogeneous isolated surface singularity, 
then $HP_0(A)$ coincides with the Jacobi ring of $A$. 
Later, in \cite{ES1}, T. Schedler and the second author 
determined $HP_0(A)$ for symmetric powers of such surfaces. Also, 
it is easy to show that if $A$ is the algebra of functions
on an irreducible affine symplectic variety $X$
of dimension $2d$, then $HP_0(A)=H^{2d}(X,F)$, 
and the paper \cite{ES2} generalizes this result to symmetric 
powers of symplectic varieties. Finally, the paper \cite{ES3}
uses the theory of ${\mathcal D}$-modules to study the space $HP_0(A)$
in the case when $A$ is the algebra ${\mathcal O}_X$ of functions 
on any affine Poisson variety $X$ over a field of characteristic zero. 
In particular, it shows that if $X$ has finitely many symplectic leaves, 
then the space $HP_0(A)$ is finite dimensional. 
This theory can also be used to calculate $HP_0(A)$ using a computer 
in the case of quotient singularities (see \cite{EGPRS}). 

The goal of this paper is to generalize
some of these results to positive characteristic. 
More precisely, we study Poisson traces 
for Poisson algebras $A_p$ 
obtained by reduction modulo primes $p$ of a finitely 
generated Poisson algebra $A$ in characteristic zero, 
for sufficiently large $p$. In this case, the space $HP_0(A_p)$ 
is typically infinite dimensional, 
but if $A$ is positively graded, then so is the space $HP_0(A_p)$, 
and one may ask for its Hilbert series, which 
converges to a rational function. 

Our first main result is the computation of 
this Hilbert series for a quasihomogeneous isolated surface singularity
in $\Bbb A^3$. Our second main result is the computation of the space 
$HP_0(A_p)$ and its Hilbert series when $A$ is the algebra of invariants in a polynomial algebra 
under a finite group (i.e., for a quotient singularity), in terms of 
``characteristic zero data''. This includes the case of symmetric powers of symplectic vector spaces
and Kleinian singularities, when we get a completely explicit answer.
The main results are proved using the theory of ${\mathcal D}$-modules 
in positive characteristic. 
We also discuss the behavior of the Hilbert series for small $p$. 

The organization of the paper is as follows. 
In Section 2 we discuss the basics on reduction to positive characteristic, 
and give some general results. In Section 3, we state our main theorems 
on the Hilbert series of $HP_0$ for isolated 
quasihomogeneous surface singularities and 
for quotient singularities in sufficiently large 
positive characteristic, 
and give a number of examples, including cones of smooth projective curves and 
symmetric powers of symplectic vector spaces and Kleinian singularities. 
In Section 4, we prove the first main theorem. 
Finally, in Section 5, we discuss some conjectures and results on 
the behavior of $HP_0$ for small $p$. 

{\bf Acknowledgments.} The authors are grateful to Roman Bezrukavnikov and 
Travis Schedler for useful discussions. In particular, we are grateful to Roman Bezrukavnikov 
for explanations regarding Lemma \ref{fg} and Example \ref{counterex}. 
The work of P.E. was partially supported by the NSF grant DMS-1000113. The work of D.J. was supported by the NSF grant DMS-1103778.  Y.C. and M.Z. 
are grateful to the PRIMES program at the MIT Mathematics Department, where this research was done. 

\section{Preliminaries}

\subsection{Reduction modulo $p$}\label{redu}

Let $F$ be a number field. Let $X$ be an affine scheme 
of finite type defined over $F$, and $A={\mathcal O}_X$ be the corresponding 
$F$-algebra. We would like to talk about reduction of $A$ 
and $X$ modulo primes. 

To this end, fix a finitely generated subring $R\subset F$ 
(for simplicity we will assume that $R$ generates $F$ as a field), and 
a finitely generated $R$-subalgebra $A_R\subset A$ which spans $A$ over $F$. 
Clearly, this is always possible (just choose a set of generators 
$a_1,...,a_n$ for $A$ over $F$, and set $A_R=R[a_1,...,a_n]\subset A$).    

For every rational prime $p$ let ${\rm Spec}(R)_p$ be the set of primes of $R$ projecting to $p$, 
i.e., of prime ideals ${\mathfrak{p}}\subset R$ such that $R/{\mathfrak{p}}$ 
is a finite field of characteristic $p$. This set is finite and nonempty for almost all $p$. 
Let $A_{\mathfrak{p}}:=A/{\mathfrak{p}}A$ be the corresponding reduction
(which is an algebra over the field $R/{\mathfrak{p}}=\Bbb F_q$), and 
let $X_{\mathfrak{p}}={\rm Spec}(A_{\mathfrak{p}})$. 

We note that this procedure depends on the choices of $R$ and $A_R$ 
in a very mild way. Namely, if $(R,A_R)$ and $(R',A_{R'})$ are two such choices, then 
almost all of the reductions $A_{\mathfrak{p}}$ are canonically the same for both cases. 

\begin{example}
What we will do is already interesting in the simplest case $F=\Bbb Q$ and $R=\Bbb Z$. 
In this case for any prime $p$ we have the usual reductions 
$A_p$, $X_p$ of $A$, $X$ modulo $p$. 
\end{example}

We will also need to talk about reductions of ${\mathcal O}$-modules and ${\mathcal D}$-modules. 
Suppose $M$ is a coherent ${\mathcal O}$-module or ${\mathcal D}$-module 
on $X$. Then we can define the reductions $M_{\mathfrak{p}}$, 
using an $A_R$-lattice $M_R$ in $M$, and again for any two choices of $R,A_R,M_R$,
almost all the reductions are canonically the same.  

\begin{remark}
If $X$ is singular, we pick a closed embedding $i: X\to V$ into a smooth affine variety, and 
use the functor $i_*$ to realize ${\mathcal D}$-modules on $X$ as ${\mathcal D}$-modules on $V$ scheme-theoretically 
supported on $X$. Then, by reduction $M_{\mathfrak{p}}$ 
of a ${\mathcal D}$-module $M$ on $X$ we mean the ${\mathcal D}$-module $i^!(i_*M)_{\mathfrak{p}}$. Since 
reduction modulo $p$ commutes with the functors $i_*,i^!$ for closed embeddings of smooth affine varieties, 
this procedure is well defined. 
\end{remark}

Let $X$ be an affine scheme of finite type over $F$ and $M$ a coherent right ${\mathcal D}$-module on $X$. 
Consider the reductions $M_{\mathfrak{p}}$ and their underived 
direct images $\pi_0(M_{\mathfrak{p}})$ under the map $\pi: X\to {\rm pt}$ of $X$ to the point.

\begin{lemma}\label{fg} 
(i) $\pi_0(M_{\mathfrak{p}})$ is a finitely generated module 
over ${\mathcal O}_{X,{\mathfrak{p}}}^p$.

(ii) If $M$ is holonomic, then there is a constant $c=c(M)$ such that 
for almost all $p$, the fibers of the module $\pi_0(M_{\mathfrak{p}})$ over 
${\mathcal O}_{X,\mathfrak{p}}^p$ have dimension $\le c(M)$. 
\end{lemma}

\begin{proof}
(i) Let $i: X\to V$ be a closed embedding of $X$ into an affine space, 
and $\pi': V\to {\rm pt}$ be the projection of $V$ to the point. 
Let ${\mathcal D}_{V,{\mathfrak{p}}}$ be the reduction 
of the algebra ${\mathcal D}_{V}$ of differential operators on $V$. 
Then $i_*M_{\mathfrak{p}}={\mathcal D}_{V,\mathfrak{p}}^n/J$, for some $n$, where $J$ is a submodule. 
Thus, 
$$
\pi_0'(i_*M_{\mathfrak{p}})=
i_*M_{\mathfrak{p}}
\otimes_{{\mathcal D}_{V,{\mathfrak{p}}}}{\mathcal O}_{V,{\mathfrak{p}}}=
{\mathcal O}_{V,{\mathfrak{p}}}^n/K,
$$ 
for some submodule $K$ over ${\mathcal O}_{V,{\mathfrak{p}}}^p$. So $\pi_0'(i_*M_{\mathfrak{p}})=\pi_0(M_{\mathfrak{p}})$ is a finitely generated 
module over ${\mathcal O}_{V,{\mathfrak{p}}}^p$. 
Moreover, if $I$ is the defining ideal of $X$, then $I^p$ 
is killed in this module, so 
$\pi_0(M_{\mathfrak{p}})$ is a finitely generated module over ${\mathcal O}_{X,\mathfrak{p}}^p$. 

(ii) As in (i), it suffices to assume that $X=V$ is an affine space, of some dimension $n$, with coordinates $x_1,...,x_n$. 
In this case, consider the fiber $(M_{\mathfrak{p}})_z$ over a geometric point 
$z$ of the Frobenius twist $(T^*V_{\mathfrak{p}})'$ (i.e., the spectrum of the $p$-center 
of ${\mathcal D}_{V,{\mathfrak{p}}}$). 

We claim that the dimension of $(M_{\mathfrak{p}})_z$ for large enough $p$ is at most $cp^n$, where $c=c(M)$
is some constant. Indeed, let $m_1,..,m_r$ be generators of $M_R$. Then 
$M_{\mathfrak{p}}$ is generated over the $p$-center of ${\mathcal D}_{V,{\mathfrak{p}}}$
by the elements $x_1^{a_1}...x_n^{a_n}\partial_1^{b_1}...\partial_n^{b_n}m_j$, where 
$a_i,b_i\in [0,p-1]$. Since the Gelfand-Kirillov dimension of $M$ is $n$, for large $p$ these elements span
a space of dimension $\le cp^n$, which implies our claim. 

Now, fibers of $\pi_0(M_{\mathfrak{p}})$ are coinvariants with respect to $\partial_1,...,\partial_n$ 
in $(M_{\mathfrak{p}})_z$ for particular points $z$. Since ${\mathcal D}_{V,{\mathfrak{p}}}$ is an Azumaya algebra, 
we see that the dimension of any fiber of $\pi_0(M_{\mathfrak{p}})$ does not exceed $c(M)$.   
\end{proof}

\subsection{The ${\mathcal D}$-module attached to a Poisson scheme}

Let us recall the main construction of \cite{ES3}, which attaches 
to an affine Poisson scheme $X$ of finite type over a field $F$ of characteristic zero
and a morphism $\phi: X\to Y$ to another scheme $Y$, 
a right ${\mathcal D}$-module $M_\phi(X)$ on $X$. 

First assume that $X$ is a smooth affine Poisson algebraic variety
over $F$. Let $Y$ be another affine variety, 
and $\phi: X\to Y$ be a morphism. 

\begin{definition}
The right ${\mathcal D}$-module $M_\phi(X)$ attached to $(X,\phi)$
is the quotient of the algebra ${\mathcal D}_X$ 
of differential operators on $X$ by the right ideal generated by 
Hamiltonian vector fields associated to regular functions pulled back from $Y$. 
\end{definition}

If $Y=X$ and $\phi={\rm Id}$, we denote $M_\phi(X)$ simply by $M(X)$. 

Now assume that $X,Y$ are affine schemes of finite type over $F$.
Let $i: X\to V$ be a closed embedding of $X$ into a smooth affine
variety $V$.

\begin{definition} 
  The right ${\mathcal D}$-module $M_\phi(X,i)$ on $V$ is the quotient of ${\mathcal D}_V$
  by the right ideal generated by regular functions on $V$ vanishing
  at $X$ and vector fields on $V$ tangent to $X$ and specializing at $X$ to 
Hamiltonian vector fields of functions pulled back from $Y$. 
The right ${\mathcal D}$-module $M_\phi(X)$ on $X$ is $i^!  M_\phi(X, i)$.
\end{definition}

One can check that this definition does not depend on $i$, up to a canonical isomorphism. 

\begin{theorem}\label{es3}(\cite{ES3}) (i) If $X$ has finitely many symplectic leaves and $\phi$ is finite then 
the ${\mathcal D}$-module $M_\phi(X)$ is holonomic. 

(ii) If $B=\OO_X,\ A=\phi^*(\OO_Y)$, and $\pi: X\to {\rm pt}$ is the projection, then the underived direct image 
$\pi_0(M_\phi(X))$ is naturally isomorphic to $B/\lbrace{A,B\rbrace}$. 
\end{theorem}

In particular, in the situation of (i), $\dim B/\lbrace{A,B\rbrace}<\infty$.  

\begin{remark} All these results extend to positive characteristic
(where in Theorem \ref{es3} (i), by a holonomic ${\mathcal D}$-module 
we mean a ${\mathcal D}$-module whose singular support is a Lagrangian subvariety).   
\end{remark}

\subsection{Poisson algebras and the zeroth Poisson homology}

Assume now that in the setting of Subsection \ref{redu}, $A_R$, and hence $A$ is a Poisson algebra, i.e., 
$X$ is a Poisson scheme. For each ${\mathfrak{p}}\in {\rm Spec}(R)_p$, 
consider the zeroth Poisson homology $HP_0(A_{\mathfrak{p}})=A_{\mathfrak{p}}/\lbrace {A_{\mathfrak{p}},A_{\mathfrak{p}}\rbrace}$. 
This is a finitely generated module over the $p$-th power algebra $A_{\mathfrak{p}}^p$
(since so is $A_{\mathfrak{p}}$, and $\lbrace {A_{\mathfrak{p}},A_{\mathfrak{p}}\rbrace}$ is a submodule).

\begin{proposition}\label{fiber}
Suppose that $X$ has finitely many symplectic leaves. 
Then for large enough $p$, the dimension of a 
fiber of the module $HP_0(A_{\mathfrak{p}})=A_{\mathfrak{p}}/\lbrace {A_{\mathfrak{p}},A_{\mathfrak{p}}\rbrace}$
over $A_{\mathfrak{p}}^p$ is bounded from above by a constant $c$ 
which depends only on the Poisson algebra $A$ (and not on $p$). 
\end{proposition} 

\begin{proof} 
Similarly to Theorem \ref{es3}, we have $\pi_0(M(X)_{\mathfrak{p}})=
A_{\mathfrak{p}}/\lbrace{A_{\mathfrak{p}},A_{\mathfrak{p}}\rbrace}$. Thus, 
Proposition \ref{fiber} follows from Lemma \ref{fg}(ii). 
\end{proof}

\begin{example}
Let $X$ be an affine symplectic variety of dimension $2n$. 
Then it is easy to show that, similarly to characteristic zero, 
for large $p$ we have $HP_0(A_{\mathfrak{p}})=H^{2n}_{DR}(X_{\mathfrak{p}})$, 
the top De Rham cohomology of the 
the reduction $X_{\mathfrak{p}}$ of $X$. Via the Cartier operator, 
this cohomology is naturally isomorphic 
to the (infinite dimensional) space $\Omega^{2n}(X_{\mathfrak{p}}')$, 
of top differential forms on the Frobenius twist of $X_{\mathfrak{p}}$. 
\end{example} 

\subsection{Graded Poisson algebras} 

For a $\Bbb Z_+$-graded vector space $Y$ define its Hilbert series 
$$
h_Y(t)=\sum_{m=0}^\infty a_m t^m,
$$
where $a_m$ is the dimension of $Y[m]$.

Now let $A$ be a $\Bbb Z_+$-graded Poisson algebra.
This means that the multiplication has degree zero, and the Poisson bracket 
is homogeneous (of some integer degree, not necessarily zero).  
Assume that all the homogeneous components of $A$
are finite dimensional. Then for any ${\mathfrak{p}}\in {\rm Spec}(R)_p$, 
one can define the Hilbert series 
$$
H_{A,\mathfrak{p}}(t)=h_{A_{\mathfrak{p}}/\lbrace{A_{\mathfrak{p}},A_{\mathfrak{p}}\rbrace}}(t)=
\sum_{m=0}^\infty b_m t^m, 
$$
where $b_m=\dim (A_{\mathfrak{p}}/\lbrace{A_{\mathfrak{p}},A_{\mathfrak{p}}\rbrace})[m]$.  
We also define $H_{A,0}(t)=h_{A/\lbrace{A,A\rbrace}}(t)$ 
to be the same Hilbert series in characteristic zero 
(i.e., over $F$). It is clear that $H_{A,\mathfrak{p}}(t)$ 
is a power series with nonnegative coefficients, which 
by Lemma \ref{fg} (i) converges to rational function of $t$. 

\begin{proposition} Suppose that $X={\rm Spec}(A)$ has finitely many symplectic leaves. 
Then there exist integers $m_1,...,m_r$ such that for large $p$ 
$$
H_{A,\mathfrak{p}}(t)=\sum_{j=1}^r t^{m_j}f_{j,{\mathfrak{p}}}(t^p).
$$ 
\end{proposition}

\begin{proof} By definition, the group $\Bbb G_m$ acts on $X$, 
and the Poisson bracket on $X$ transforms under this action by a character. 
Let $i: X\to V$ be a closed embedding of $X$ into a vector space,
which is equivariant with respect to $\Bbb G_m$. Then the action of  
$\Bbb G_m$ on $X$ and $V$ defines its action on the vector space (of global sections of) 
$i_*M(X)$ (i.e., the ${\mathcal D}$-module $i_*M(X)$ is weakly equivariant, see \cite{ES3}). 
Let $E$ be the Euler field on $V$, and $\rho(E)$ be the operator corresponding 
to $E$ under this action. Also, let $\tau(E)$ be the usual action of $E$ on $i_*M(X)$ (as a differential operator). 
Then $L:=\rho(E)-\tau(E)$ is an endomorphism of the ${\mathcal D}$-module $i_*M(X)$.  
Since $X$ has finitely many symplectic leaves, $M(X)$ is a holonomic ${\mathcal D}$-module, so 
it has finite length. Thus, so does $i_*M(X)$. This means that the operator $L$ acts 
on $i_*M(X)$ semisimply, with finitely many integer eigenvalues $m_1,...,m_r$
(as ${\rm End}(M(X))$ is a finite dimensional algebra). This remains true upon reduction
modulo a large prime. Thus, the eigenvalues of $L$ on $HP_0(A_{\mathfrak{p}})$, which is the underived direct image 
of $i_*M(X)_{\mathfrak{p}}$, are $m_1,...,m_r$. This implies the required statement, as $\tau(E)$ acts on the direct image 
by zero.   
\end{proof}

\begin{question}\label{ratfu}  
Is it true that the functions $f_{j,{\mathfrak{p}}}$ are in fact independent on ${\mathfrak{p}}$ (for large enough $p$)?  
\end{question} 

The answer to this question is positive for all the examples
considered below. 

\begin{remark}\label{counterex} 
We note, however, that the answer is negative for 
the underived direct image to the point of 
an arbitrary weakly $\Bbb G_m$-equivariant holonomic 
${\mathcal D}$-module. For example, let $M$ be the (left) ${\mathcal D}$-module on 
$\Bbb A_1$ defined by the differential equation 
$x\partial \psi=\sqrt{-1}\psi$
(it is defined over $F=\Bbb Q(\sqrt{-1})$).
This ${\mathcal D}$-module is monodromic, in 
particular, weakly $\Bbb G_m$-equivariant. 
Then, if $-1$ is a nonsquare modulo $p$
(i.e., $p=4k+3$), then $\pi_0(M_{\mathfrak{p}})=0$. 
On the other hand, if $-1$ is a square 
modulo $p$ (i.e., $p=4k+1$), then 
$\pi_0(M_{\mathfrak{p}})\ne 0$. 

One can modify this example to 
be defined over $\Bbb Q$ by taking 
the ${\mathcal D}$-module defined by the differential equations 
$x\partial \psi_1=\psi_2$, $x\partial \psi_2=-\psi_1$
(it is defined over $F=\Bbb Q$ and splits as $M\oplus \bar M$ 
after a field extension).

We note that a ${\mathcal D}$-module similar to $M$ can occur as a subquotient 
(and even direct summand) of $M(X)$ for a suitable $X$. 
Namely, by analogy with Example 4.11 of \cite{ES3},
let $Z$ be the cone of an elliptic curve $Q(x,y,z)=0$ 
(where $Q$ is a homogeneous cubic polynomial), equipped
with the standard Poisson structure, given generically by the symplectic form 
$\omega=\frac{{\bold d}x\wedge {\bold d}y\wedge {\bold d}z}{{\bold d}Q}$. 
This Poisson structure has
degree zero, so the Euler derivation $E: \OO_Z\to \OO_Z$ is Poisson.
Let $X=Z\times \Bbb A^1\times (\Bbb A^1\setminus 0)$, where $\Bbb A^1\times (\Bbb A^1\setminus 0)$ 
has coordinates $q,p$ with $\lbrace{p,f\rbrace}=\sqrt{-1} pEf$, $\lbrace{q,f\rbrace}=0$, 
$f\in \OO_Z$, and $\lbrace{p,q\rbrace}=p$. Then $X$ is a Poisson
variety with two symplectic leaves: the open four-dimensional leaf, and the 
two-dimensional leaf $S=0\times \Bbb A^1\times (\Bbb A^1\setminus 0)$. 

By analogy of Proposition 4.12 of \cite{ES3}, 
one can prove the following proposition. 

Let $i: S\to X$ be the closed embedding. 

\begin{proposition}
$M(X)$ maps surjectively onto $i_*(N_0\oplus 3N_1\oplus 3N_2\oplus N_3)$, where 
$N_m$, $m\ge 0$, is the quotient of the algebra of differential operators in $p,q$ 
by the right ideal generated by $q\partial_q-\sqrt{-1}m, \partial_p$. 
\end{proposition} 

However, $X$ is not conical, so 
this example does not allow us to answer Question \ref{ratfu}. 
\end{remark}

\section{Main results} 

\subsection{Quasihomogeneous isolated surface singularities} 

Let $F$ be a field of characteristic zero, and $Q\in F[x,y,z]$ be a quasihomogeneous polynomial
of degree $d$, where $\deg(x)=a$, $\deg(y)=b$, $\deg(z)=c$.  
Let $X\subset \Bbb A^3$ be the surface $Q=0$, 
and assume that $0$ is an isolated 
singular point of $X$ (in particular, $Q$ is irreducible). 

Let $A={\mathcal O}_X$. Then $A$ is a Poisson algebra with 
Poisson bracket defined by
$$
\lbrace{x,y\rbrace}=Q_z,\ \lbrace{y,z\rbrace}=Q_x,\ \lbrace{z,x\rbrace}=Q_y.
$$

\begin{proposition}\label{jacring}\cite{AL}
The subspace $\lbrace{A,A\rbrace}$ is the ideal in $A$ generated by 
$Q_x,Q_y,Q_z$. Thus, the space $HP_0(A)$ is naturally isomorphic to the Jacobi ring 
$J(Q):=F[x,y,z]/(Q_x,Q_y,Q_z)$, and 
$$
H_{A,0}(t)=\frac{(1-t^{d-a})(1-t^{d-b})(1-t^{d-c})}{(1-t^a)(1-t^b)(1-t^c)}
$$
(which is a polynomial in $t$). 
\end{proposition}

\begin{proof} Note that $Q=\frac{1}{d}(axQ_x+byQ_y+czQ_z)\in (Q_x,Q_y,Q_z)\subset F[x,y,z]$, so 
$A/(Q_x,Q_y,Q_z)=J(Q)$. The formula for the Hilbert series of the Jacobi ring 
follows from the fact that $Q_x,Q_y,Q_z$ is a regular sequence in $F[x,y,z]$, as our singularity is isolated. 
Thus, it suffices to prove the first statement. 

It is clear that $\lbrace{A,A\rbrace}$ is contained in $(Q_x,Q_y,Q_z)$, so 
our job is to establish the opposite inclusion. This means that 
given a 2-form $h$ on ${\Bbb A}^3$, we must represent $h\wedge {\bold d}Q$ as $Qu+{\bold d}v\wedge {\bold d}Q$, where $u$ is a 3-form and $v$ a 1-form.

Let $\omega$ be the standard volume form ${\bold d}x\wedge {\bold d}y\wedge {\bold d}z$. 
Then $Q\omega={\bold d}E\wedge {\bold d}Q$, where
$$
E:=ax{\bold d}y\wedge {\bold d}z+by{\bold d}z\wedge {\bold d}x+cz{\bold d}x\wedge {\bold d}y.
$$
 So if $u=f\omega$, we get
$$
h\wedge {\bold d}Q=d^{-1}fE\wedge {\bold d}Q+{\bold d}v\wedge {\bold d}Q.
$$ 
So it suffices to solve for $f$ the equation
$$
h=d^{-1}fE+{\bold d}v,
$$ 
i.e., 
$$
{\bold d}h=d^{-1}{\bold d}(fE).
$$ 
Assume that ${\bold d}h=g\omega$, where $g$ is quasihomogeneous of degree $n$.
Then we need to solve $g=d^{-1}(a+b+c+n)f$
(since ${\bold d}E=(a+b+c)\omega$). So we can take $f=\frac{d}{a+b+c+n}g$, and we are done.  
\end{proof}

\begin{remark}\label{posch}
In fact, this proof partially works also in positive characteristic. 
Namely, it shows that over a field of characteristic $p$ relatively prime to $d$, 
$\lbrace{A,A\rbrace}[m]=(Q_x,Q_y,Q_z)[m]$ in all degrees $m$ except $m=pk+d-a-b-c$, where $k$ is a positive 
integer. 
\end{remark}

Our first main result is the following theorem. Let 
$$
g(z):=z^{a+b+c-d}\frac{1-z^d}{(1-z^a)(1-z^b)(1-z^c)}.
$$ 
Also, for any rational function $\phi(z)$, let 
$\phi_-(z),\phi_+(z)$ be the sum of all the terms 
of the Laurent expansion of $\phi(z)$ at $0$ 
of negative and positive degree, respectively, and let $\phi_0$ 
be the constant term of this expansion. 

Since we will do reduction modulo $p$, we assume that $F$ is a number field. 

\begin{theorem}\label{th1} For sufficiently large $p$,
the function $H_{A,{\mathfrak{p}}}(t)$ for ${\mathfrak{p}}$ sitting over $p$ 
is given by the formula 
\begin{equation}\label{mainform}
H_{A,{\mathfrak{p}}}(t)=\frac{(1-t^{d-a})(1-t^{d-b})(1-t^{d-c})}{(1-t^a)(1-t^b)(1-t^c)}+t^{d-a-b-c}f(t^p),
\end{equation}
where 
$$
f(z)=g(z)-g_-(z)+g_-(z^{-1})-g_0.
$$
\end{theorem}

Theorem \ref{th1} was discovered by analyzing results of MAGMA 
computations, \cite{M}. It is proved in the next section. 

\begin{example}\label{kle} Suppose $Q=0$ is a Kleinian singularity. 
In this case, $d=2h$, where $h$ is the Coxeter number of the corresponding 
root system under McKay's correspondence, 
and $c=h=a+b-2$. So $a+b+c-d=2$, and thus $g_-=g_0=0$. 
Also, it is known that $H_{A,0}(t)=\sum_{i=1}^r t^{2(m_i-1)}$, 
where $m_i$, $i=1,...,r$ are the corresponding exponents. 
So we have 
$$
H_{A,{\mathfrak{p}}}(t)=\sum_{i=1}^r t^{2(m_i-1)}+t^{2p-2}\frac{1+t^{ph}}{(1-t^{pa})(1-t^{pb})}.
$$
\end{example}

\begin{example} Suppose $a=b=c=1$, i.e. $Q$ defines a smooth 
plane projective curve of degree $d$. Let $\chi=(3-d)d$ 
be the Euler characteristic of this curve. Then it follows from Theorem \ref{th1} and a 
direct computation that 
$$
H_{A,{\mathfrak{p}}}(t)=\frac{(1-t^{d-1})^3}{(1-t)^3}+t^{d-3}f(t^p),
$$
where 
$$
f(z)=\frac{1-z^d}{(1-z)^3}-\frac{\chi z}{1-z}-1.
$$
\end{example}

\subsection{Quotient singularities}

Let $V$ be a symplectic vector space over a number field $F$, 
and $B=F[V]$. Let $G$ 
be a finite subgroup of $Sp(V)$. 
Let $X=V/G$, and $A={\mathcal O}_X=F[V]^G$.
Let $S$ be the set of parabolic subgroups of $G$, i.e., 
stabilizers 
of some points of $V$. For every subgroup $K\in S$, let $V^K$ 
be the fixed subspace of $K$; its generic point has stabilizer $K$. 
Let $\overline{V^K}$ be the unique $K$-invariant complement of $V^K$ in $V$. 
This is a representation of $K$ which does not have nonzero invariants. 
Let $B_K=F[\overline{V^K}]$, and $A_K=F[\overline{V^K}]^K$.
Also, let $N(K)$ be the normalizer of $K$ in $G$, and $N_0(K):=N(K)/K$. 

Our second main result is a characterization of the space 
$B_{\mathfrak{p}}/\lbrace{A_{\mathfrak{p}},B_{\mathfrak{p}}\rbrace}$
for large $p$. 

\begin{theorem}\label{th2}
For large enough $p$, the space 
$B_{\mathfrak{p}}/\lbrace{A_{\mathfrak{p}},B_{\mathfrak{p}}\rbrace}$
is isomorphic, as a $\Bbb Z_+$-graded $G$-module, to 
the space 
$$
\oplus_{K\in S}(B_K/\lbrace{A_K,B_K\rbrace})_{\mathfrak{p}}\otimes 
H^{\dim V^K}(V^K_{\mathfrak{p}}),
$$
where the last factor is the top De Rham cohomology of $V^K$ in 
characteristic $p$. 
\end{theorem}

\begin{proof}
Let $\phi: V\to V/G$ be the projection mapping, and 
let $M_\phi(V)$ be the ${\mathcal D}$-module attached to 
$V$ and $\phi$ defined in \cite{ES3}. 
By Theorem 4.13 in \cite{ES3}, we have 
$$
M_\phi(V)=\oplus_{K\in S}(B_K/\lbrace{A_K,B_K\rbrace})\otimes 
\delta_{V^K},
$$
where $\delta_{V^K}$ is the direct image of $\Omega$ from $V^K$ to 
$V$. Now, if $\pi: V\to {\rm pt}$ is the projection, then 
$$
\pi_0((\delta_{V^K})_{\mathfrak{p}})=H^{\dim V^K}(V^K_{\mathfrak{p}}).
$$ 
This implies the required statement. 
\end{proof}

Taking $G$-invariants in Theorem \ref{th2}, we obtain

\begin{corollary}\label{cor1}
For large enough $p$, the space 
$HP_0(A_{\mathfrak{p}})$
is isomorphic, as a $\Bbb Z_+$-graded space, to 
the space 
$$
\oplus_{K\in S/G}\left(HP_0(A_K)_{\mathfrak{p}}\otimes 
H^{\dim V^K}(V^K_{\mathfrak{p}})\right)^{N_0(K)}.
$$ 
\end{corollary}

Let us derive an explicit formula for the Hilbert series of this space. 
To this end, for every irreducible representation 
$\pi$ of $N_0(K)$, let $\psi_\pi(t)$ be the Hilbert series 
of the graded space $\Hom_{N_0(K)}(\pi,HP_0(A_K))$.
Also, let $\eta_\pi(z)$ be the Hilbert series 
of $(\pi\otimes F[V^K])^{N(K)}$. 
This is data from characteristic zero which we will assume known. 
Then, using the identification of $H^{\dim V^K}(V_{\mathfrak{p}}^K)$ 
with top differential forms on the Frobenius twist of $V_{\mathfrak{p}}^K$ 
using the Cartier operator, we have:

\begin{corollary}\label{cor2}
For large enough $p$ we have 
\begin{equation}\label{mainform1}
H_{A,\mathfrak{p}}(t)=\sum_{K\in S/G}\sum_{\pi\in {\rm Irrep}N_0(K)}
t^{(p-1)\dim V^K}\psi_\pi(t)\eta_\pi(t^p).
\end{equation}
\end{corollary}

\begin{example} \label{klei} Consider the case of a Kleinian singularity.
This means that $V=L$ is a 2-dimensional space, 
and $G=\Gamma$ is a finite subgroup of $SL(L)$. 
In this case, the only possibility for $K$ 
besides $K=\Gamma$ is $K=1$, in which case 
$\overline{L^K}=0$, and $N(K)=N_0(K)=\Gamma$. 
Thus, from Corollary \ref{cor2} we get 
$$
H_{A,{\mathfrak{p}}}(t)=\sum_{i=1}^r t^{2(m_i-1)}+
t^{2p-2}\eta_{\rm triv}(t^p),
$$
where $\eta_{\rm triv}(z)$ is the Hilbert series of $F[L]^\Gamma$. 
It is known that 
$$
\eta_{\rm triv}(z)=\frac{1+z^h}{(1-z^a)(1-z^b)},
$$
where $a$ and $b$ are the degrees of the primary invariants 
(the third generating invariant which is the Poisson bracket of 
the first two, 
has degree $h=a+b-2$). Thus, we get 
$$
H_{A,{\mathfrak{p}}}(t)=\sum_{i=1}^r t^{2(m_i-1)}+
t^{2p-2}\frac{1+t^{ph}}{(1-t^{pa})(1-t^{pb})},
$$
the same answer as in Example \ref{kle}. 
\end{example}

\begin{example}\label{sympow}
Let $L$ be a symplectic vector space of dimension $2d$ over $F$, 
and $V=L^n$. Take $G=S_n$ permuting the copies of $L$. 
Let $A=F[V]^G$. 
In this case, the possible groups $K$ up to conjugation are 
$K=P_\lambda=S_{\lambda_1}\times S_{\lambda_2}\times...\times S_{\lambda_k}$, 
with $\lambda_1\ge...\ge \lambda_k$ (i.e., they are parametrized 
by partitions $\lambda$ of $n$). So Corollary \ref{cor1} implies the 
following result. 

Let $m_j(\lambda)$ is the number of times $j$ occurs among $\lambda_i$. 

\begin{proposition} 
For large $p$ the space $HP_0(A_{\mathfrak{p}})$ is given by the formula 
\begin{equation}\label{hp0l}
HP_0(A_{\mathfrak{p}})=\oplus_\lambda \otimes_j {\rm Sym}^{m_j(\lambda)}H^{2d}(L_{\mathfrak{p}}).
\end{equation}
In particular, $H_{A,{\mathfrak{p}}}(t)=a_{n,p}(t)$, where 
$$
\sum_{n\ge 0}a_{n,p}(t)s^n=\prod_{j\ge 0}\prod_{k\ge 0}(1-s^{j+1}
t^{2d(p-1)+kp})^{-\binom{2d+k-1}{k}}.
$$ 
\end{proposition}
\end{example}

\begin{remark}
In fact, it is easy to 
see that formula (\ref{hp0l}) holds for any affine symplectic 
variety $L$ of dimension $2d$, with the same proof as in the vector space case. 
This gives a generalization Theorem 1.1.1 of \cite{ES2} to the case of positive characteristic. 
\end{remark}

\begin{example} Now consider symmetric powers 
of Kleinian singularities. Let $V=L^n$ and $G=S_n\ltimes \Gamma^n$, 
where $L$ is as in Example \ref{klei}. Let $A=F[V]^G$.  
Then, generalizing Example \ref{klei} (corresponding to $n=1$), 
similarly to Example \ref{sympow} 
from Corollary \ref{cor2} we get the following proposition. 
\begin{proposition} 
We have $H_{A,{\mathfrak{p}}}(t)=a_{n,p}(t)$, where 
$$
\sum_{n\ge 0}a_{n,p}(t)s^n=\prod_{j\ge 0}\left(\prod_{i=1}^r(1-s^{j+1}t^{2(m_i-1+jh)})^{-1}
\prod_{k\ge 0}(1-s^{j+1}t^{2d(p-1)+kp})^{-d_k}\right),
$$ 
where 
$$
\sum_{k\ge 1}d_kz^k=\frac{1+z^h}{(1-z^a)(1-z^b)}.
$$
\end{proposition}
\end{example}

\section{Proof of Theorem \ref{th1}}

By rescaling the degrees, we may assume that 
$GCD(a,b,c)=1$. Let us make this assumption 
throughout the proof. 

It follows from Remark \ref{posch} that for large $p$ 
$$
H_{A,{\mathfrak{p}}}(t)=\frac{(1-t^{d-a})(1-t^{d-b})(1-t^{d-c})}{(1-t^a)(1-t^b)(1-t^c)}+t^{d-a-b-c}f(t^p),
$$
where $f$ is a rational function such that $f(0)=0$. 
So it remains to compute $f$. We will compute $f$ using the formula $HP_0(A_{\mathfrak{p}})=\pi_0(M(X)_{\mathfrak{p}})$. 

The surface $X$ consists of two symplectic leaves -- the point $0$ and its complement $X^\circ$. 
Thus, $M(X)$ is an extension (on both sides) of the irreducible ${\mathcal D}$-module $IC_X$ by finitely many 
copies of $\delta_0$. The $\delta_0$-constituents cannot contribute to the function $f$, so 
for the purposes of computing $f$ we can replace $M(X)$ by the complex of ${\mathcal D}$-modules $j_*\Omega$, where $j: X^\circ \to X$ is the 
open embedding. Indeed, $j_*\Omega$ differs from $M(X)$ by finitely many 
copies of $\delta_0$. 

Thus, we need to compute the components of degrees $pk+d-a-b-c$ (for $k>0$) in $H^0(\pi_*(j_*\Omega)_{\mathfrak{p}})$. 
Since we are interested in large positive degrees only, we can replace $(j_*\Omega)_{\mathfrak{p}}$ by 
$j_*\Omega_{\mathfrak{p}}$. So we need to compute 
components of degrees $pk+d-a-b-c$ in $H^0(\pi_*j_*\Omega_{\mathfrak{p}})=H^0(\tilde \pi_*\Omega_{\mathfrak{p}})$,
where $\tilde\pi: X^\circ\to {\rm pt}$ is the projection. 

This cohomology can be computed by first modding out by the action of $\Bbb G_m$ and then projecting 
to the point. Namely, let $C_X=X^\circ/\Bbb G_m$ be the corresponding weighted projective curve. 
This curve is in general an orbifold, but in our computation it behaves like a smooth curve (with appropriate 
modifications to take into account orbifold points). We have the maps $\xi: X^\circ \to C_X$ and 
$\theta: C_X\to {\rm pt}$, and $\tilde \pi_*=\theta_* \circ \xi_*$. So we need to compute the components of 
degrees $pk+d-a-b-c$ in $H^0(\theta_* \xi_*\Omega_{\mathfrak{p}})$.  

Recall that if ${\mathcal L}$ is any line bundle on a smooth variety $Y$ in characteristic $p$, then 
the line bundle ${\mathcal L}^{\otimes p}$ carries a canonical flat connection. 
Namely, if $\nabla$ is any connection on ${\mathcal L}$ defined on an affine open set $U\subset Y$, then 
we have a natural connection $\nabla^{(p)}$ on ${\mathcal L}^{\otimes p}$ over $U$. 
We claim that this connection is independent of $\nabla$. Indeed, if 
$\nabla_2=\nabla_1+a$, where $a$ is a 1-form, then 
$\nabla_2^{(p)}=\nabla_1^{(p)}+pa=\nabla_1^{(p)}$. 
This implies that $\nabla^{(p)}$ is defined globally and is canonical.
Also, it is flat, since the curvature of $\nabla^{(p)}$ is $p$ times the curvature of $\nabla$. 

Since $\xi$ is a principal $\Bbb G_m$-bundle over $C_X$, 
and $H^0_{DR}(\Bbb G_m,\Bbb F_q)=\Bbb F_q[z^p,z^{-p}]$,
$H^1_{DR}(\Bbb G_m,\Bbb F_q)=\Bbb F_q[z^p,z^{-p}]dz/z$, 
we have 
$$
H^0(\xi_*\Omega_{\mathfrak{p}})=H^{-1}(\xi_*\Omega_{\mathfrak{p}})=\oplus_{k\in \Bbb Z}{\mathcal L}^{\otimes pk}\otimes \Omega_{\mathfrak{p}},
$$
where ${\mathcal L}$ is the tautological line bundle on $C_X$, and all the other cohomology groups are zero 
(here to simplify notation, we denote the reduction of 
$C_X$ to characteristic $p$ also by $C_X$). 

Note that by the above explanation, 
the right hand side has a canonical structure of a right ${\mathcal D}$-module on $C_X$. We claim that this is the 
${\mathcal D}$-module structure of the direct image; in other words, the Gauss-Manin connection on ${\mathcal L}^{\otimes pk}$ coincides with 
the canonical connection defined above. Indeed, this follows from the fact that it is so for the trivial line bundle on a formal disk. 

Our job is to compute the large positive degree 
components of the zeroth cohomology of $\theta_*$ of this complex of ${\mathcal D}$-modules. 

Since we are only interested in large positive degrees, it suffices to restrict ourselves to $k>0$. Consider 
$\theta_*({\mathcal L}^{\otimes pk}\otimes \Omega_{\mathfrak{p}})$. This is nothing but the cohomology of the 
complex 
$$
0\to H^0(C_X,{\mathcal L}^{\otimes pk})\to H^0(C_X,{\mathcal L}^{\otimes pk}\otimes \Omega_{\mathfrak{p}})\to 0,
$$
where the map is the De Rham differential ${\bold d}_k$ corresponding to the canonical connection on ${\mathcal L}^{\otimes pk}$; 
this cohomology is concentrated in degrees $0$ and $-1$. Thus, to compute the zeroth cohomology of 
$\theta_*\xi_*\Omega_{\mathfrak{p}}$ in large positive degrees, we only need to consider 
$H^0(\xi_*\Omega_{\mathfrak{p}})$ (and not $H^{-1}$). Specifically, we have   
$$
H^0(\theta_*\xi_*\Omega_{\mathfrak{p}})[pk+d-a-b-c]={\rm Coker}{\bold d}_k.
$$ 

It remains to compute $b_k:=\dim{\rm Coker}{\bold d}_k$. We have 
$$
\dim {\rm Coker}{\bold d}_k=\dim{\rm Ker}{\bold d}_k-
\dim H^0(C_X,{\mathcal L}^{\otimes pk})+\dim H^0(C_X,{\mathcal L}^{\otimes pk}\otimes \Omega_{\mathfrak{p}}).
$$
Now, the flat sections of the canonical connection on ${\mathcal L}^{\otimes pk}$ 
are $p$-th powers of sections of 
${\mathcal L}^{\otimes k}$, so 
$\dim {\rm Ker} {\bold d}_k=\dim H^0(C_X,{\mathcal L}^{\otimes k})$.
So 
$$
b_k=\dim H^0(C_X,{\mathcal L}^{\otimes k})-
\dim H^0(C_X,{\mathcal L}^{\otimes pk})+\dim H^0(C_X,{\mathcal L}^{\otimes pk}\otimes \Omega_{\mathfrak{p}}).
$$

Consider the symplectic form on $X^\circ$. It is given by the formula 
$$
\omega=\frac{{\bold d}x\wedge {\bold d}y\wedge {\bold d}z}{{\bold d}Q},
$$ 
so it has homogeneity degree $a+b+c-d$. This implies that the line bundle 
$\Omega\otimes {\mathcal L}^{a+b+c-d}$ on $C_X$ is trivial, i.e., 
$\Omega={\mathcal L}^{d-a-b-c}$. (Indeed, every local section 
$s$ of ${\mathcal L}^{d-a-b-c}$ gives rise to a dilation-invariant 2-form 
$s\omega$ on $X^\circ$, whose pushforward to $C_X$ is a 1-form on 
$C_X$). 

Thus, we get 
$$
b_k=a_k-a_{pk}+a_{pk+d-a-b-c},
$$
where $a_k=\dim H^0(C_X,{\mathcal L}^{\otimes k})$, i.e. 
$$
\sum_{k=0}^\infty a_kz^k=h_A(t)=\frac{1-z^d}{(1-z^a)(1-z^b)(1-z^c)}
$$  

The rest of the proof is combinatorial. 
We can rewrite the formula for $b_k$ as 
$$
b_k=a_k-c_{pk}, 
$$
where $c_l$ is the $l$-th coefficient of 
the Laurent expansion at $t=0$ of the function 
$$
u(z)=\frac{(1-z^d)(1-z^{a+b+c-d})}{(1-z^a)(1-z^b)(1-z^c)}.
$$

\begin{lemma}\label{indepp}
(i) The function $u(z)$ admits a representation 
$$
u(z)=v(z)+\sum_{i=1}^N \alpha_i \frac{\beta_i z}{1-\beta_i z},
$$
where $\alpha_i\in \Bbb Q$, $\beta_i$ are roots of unity of 
degree $LCM(a,b,c)$, and $v(z)=g_-(z^{-1})-g_-(z)-g_0+1$. 

(ii) If $r$ is coprime to $LCM(a,b,c)$, then for any positive $k$, 
$c_{rk}=c_k-s_k+s_{rk}$, where $\sum_{k\ge 1}s_kz^k=g_-(z^{-1})$. 
\end{lemma}

\begin{proof}
(i) First of all, the function $u(z)$ has only simple poles at $z\ne 0,\infty$.
Indeed, since $GCD(a,b,c)=1$, it suffices to consider a 
point $z=\zeta$ which is a root of unity of degree $m$ dividing 
two of the numbers $a,b,c$, e.g., $a$ and $b$.  
In this case, the denominator of $u$ has a double zero, so it suffices 
to show that the numerator has a zero. To see this, recall that $Q=0$
is an isolated singularity, so $Q$ must contain a term involving only 
the variables $x,y$ (otherwise $Q$ would be reducible). This implies that $m$ divides $d$, 
so the numerator has a zero, and the pole at $\zeta$ is simple. 

This implies that $u$ has a required representation with some $v,\alpha_i$ 
(and this representation is obviously unique). It remains to show that 
$\alpha_i\in \Bbb Q$ and prove the formula for $v(z)$. 

The rationality of $\alpha_i$ is the statement 
that the residues of $u(z)dz/z$ at $z\ne 0$ are 
rational. To see this, we should consider a pole $\zeta$, 
and we can clearly assume that $\zeta\ne 1$ 
(since the residue at $\zeta=1$ is obviously rational). 
Assume first that $\zeta^a=1$ 
but $\zeta^b,\zeta^c,\zeta^d,\zeta^{a+b+c-d}\ne 1$. 
In this case, since the character of the Jacobi ring is a polynomial, we have 
that either $\zeta^{d-b}=1$ or $\zeta^{d-c}=1$. 
Assume that $\zeta^{d-b}=1$. Then $\zeta^{a+b+c-d}=\zeta^c$. 
So in the expression of the residue, two factors in the numerator cancel with 
two factors in the denominator, and the residue equals $-1/a$. 
Similarly, if $\zeta^a=1$ and $\zeta^b=1$ then $\zeta^d=1$, so 
$\zeta^{a+b+c-d}=\zeta^c$, and after canceling a factor 
in the numerator and denominator we get that the residue is $-\frac{d}{ab}$.

Now, $u(z)+u(z^{-1})=0$, so we get that $v(z)+v(z^{-1})$ 
is a function regular at zero and infinity. Hence it is a constant.
Clearly, $v_-(z)=-g_-(z)$, so $v(z)=g_-(z^{-1})-g_-(z)+C$,
and it is easy to see that $C=-g_0+1$. So (i) is proved.  

(ii) From (i) we have 
$$
\sum_{k\ge 1}(c_k-s_k)z^k=\sum_{i=1}^N \alpha_i \frac{\beta_i z}{1-\beta_i z}.
$$
So $c_k-s_k=\sum_{i=1}^N \alpha_i\beta_i^k$. 
Now, the right hand side lies in the cyclotomic field 
$\Bbb Q(\zeta)$, where $\zeta=e^{2\pi i/LCM(a,b,c)}$. This field has an automorphism 
$\zeta\to \zeta^r$. Since $\alpha_i$ are rational, applying this automorphism, we get 
$$
c_k-s_k=\sum_{i=1}^N \alpha_i\beta_i^{rk}=c_{rk}-s_{rk},
$$
as desired. 
\end{proof}

Now we can finish the proof of the theorem. By Lemma \ref{indepp}(ii), 
for large $p$, $c_{pk}=c_k-s_k$ (as $s_{pk}=0$). 
So we find that 
$$
b_k=a_k-c_{pk}=a_k-c_k+s_k.
$$
These are nothing but the positive degree coefficients of the function 
$g(z)+g_-(z^{-1})$, which implies that 
$$
\sum_{k\ge 0}b_kz^k=g(z)+g_-(z^{-1})-g_-(z)-g_0,
$$
as claimed in the theorem. Theorem \ref{th1} is proved. 

\section{Results and conjectures for small $p$}

The results we have given so far hold for sufficiently large $p$, without any 
specification on how large. 
In fact, in the setting of the previous subsections, the question how large 
$p$ should be does not even make sense, because this depends on the choice of the
reduction procedure modulo $p$. 
Nevertheless, in specific examples there are natural ways to 
make sense of lower bounds on $p$. The goal of this section is to discuss some of these examples.
Most of our statements are conjectural, based on computer 
evidence (coming from computation in MAGMA, \cite{M} and intuition). 

\subsection{Quasihomogeneous isolated surface singularities} 

Let $A$ be the algebra of regular functions on a 
quasihomogeneous isolated surface singularity $X$ given by the equation 
$Q(x,y,z)=0$ over a field $\Bbb F$ 
of characteristic $p$. 

\begin{conjecture}\label{lowbound1}
Formula (\ref{mainform}) of Theorem \ref{th1} holds 
for $h_{A/\lbrace{A,A\rbrace}}(t)$ as long as \linebreak $p>2d-a-b-c$. In particular, for Kleinian singularities 
the formula holds for $p>h$, and for cones of smooth projective curves of degree $d$ 
it holds for $p>2d-3$.   
\end{conjecture}

\begin{remark}
The motivation for this conjecture is that for $p>2d-a-b-c$, the degrees 
occurring in the second summand of formula (\ref{mainform})
are greater than the degrees in the first summand, so the corresponding spaces ``don't interact''. 
If $p<2d-a-b-c$, the formula sometimes fails for a trivial reason: 
some of the coefficients $b_m$ for small $m$ 
are bigger than $\dim A[m]$.  
\end{remark}

Note that we have the following weaker statement, which holds for all $p$. 

\begin{proposition} In the above setting, 
formula (\ref{mainform}) holds for all $p$ 
up to adding a polynomial. 
\end{proposition}

\begin{proof}
Consider $HP_0(A)$ as a module over $A^p$. 
Using the ${\mathcal D}$-module interpretation of $HP_0(A)$ from 
\cite{ES3} (which is valid in characteristic $p$), 
we see that outside of the singularity, 
$HP_0(A)$ is the canonical sheaf $\Omega(X')$ on the Frobenius twist of $X$. 
Now, $\Omega(X')$ is isomorphic to $A^p$ by multiplication by the symplectic
for $\omega=\frac{{\bold d}x\wedge {\bold d}y\wedge {\bold d}z}{{\bold d}Q}$, 
which has degree $a+b+c-d$. So the Hilbert series 
of $\Omega(X')$ is $g(t^p)$, 
where $g(z)$ is as in Theorem \ref{th1},
and $h_{HP_0(A)}(t)$ differs from this function by adding a polynomial
(since the two modules are the same outside of the singular point).  
\end{proof}

\subsection{Quotient singularities}

Let $V$ be a symplectic vector space over a field $\Bbb F$ 
of characteristic $p$, and 
$G$ a finite group of order coprime to $p$ 
acting faithfully on $V$. Let $A$ be the algebra of regular functions 
on $V/G$. 

\begin{conjecture}\label{lowbound2}
Formula (\ref{mainform1}) holds for $h_{A/\lbrace{A,A\rbrace}}(t)$
if $p>\frac{1}{2}D+1$, where $D$ is the maximal degree
of a nonzero element in $HP_0(\Bbb F[\overline{V^K}]^K)$ for $K\in S$. 
\end{conjecture}

\begin{remark}
In the case of Kleinian singularities, $D=2(h-2)$, so 
the condition is $p>h-1$, i.e. $p\ge h$. 
This is equivalent to the condition $p>h$, since 
the only case when $h$ is a prime is type $A$, and in this case $h$ is 
the order of $G$, while $p$ must be relatively prime to the order of $G$. 
\end{remark}

\subsection{Kleinian singularities}

\begin{theorem}\label{kleini}
Conjectures \ref{lowbound1} and \ref{lowbound2} 
hold for Kleinian singularities. In other words, formula 
(\ref{mainform}) holds if $p>h$. 
\end{theorem}

\begin{proof}
The type $A_{n-1}$ case is easy. In this case, 
$A$ is the subalgebra of $\Bbb F[u,v]$ 
(with the generator of $\Gamma=\Bbb Z_n$ acting on $u,v$ by
$(u,v)\mapsto (\zeta u,\zeta^{-1} v)$, where $\zeta$ is a 
primitive $n$-th root of unity in $\Bbb F$)  
generated by $uv$, $u^n$, and $v^n$. So,  
if $p>n$, then it is easy to show by a direct calculation that 
$$
A/\lbrace{A,A\rbrace}=\oplus_{m=0}^{n-2} \Bbb F(uv)^m\oplus (uv)^{p-1}\Bbb F[u^{pn},v^{pn},(uv)^p],
$$
which implies the required statement. 
So let us consider the D and E cases, for which the representation $L$ of $\Gamma$ is irreducible. 

Let $\Omega^j=\Omega^j(X^\circ)$ be the space of differential $j$-forms 
on $X^\circ=X \setminus 0$, where $X=L/\Gamma$ is a Kleinian singularity ($0\le j\le 2$). 
Let $\overline{\Omega^j}\subset \Omega^j$ be the subspace of those forms which 
are obtained by restriction from $\Bbb A^3$. 

It is clear that $\Omega^0=A$ (since any regular function on $L\setminus 0$ extends to $L$). 
Also, we have a natural isomorphism $\xi: A\to \Omega^2$ given by $f\mapsto f\omega$. 
The space $\overline{\Omega^2}$ is by definition spanned by elements of the form 
$f_1{\bold d}f_2\wedge {\bold d}f_3$, where $f_1,f_2,f_3\in A$. 
Thus, the preimage $\xi^{-1}(\overline{\Omega^2})$ is spanned by elements 
$f_1\lbrace{f_2,f_3\rbrace}$, i.e. $\xi^{-1}(\overline{\Omega^2})=A\lbrace{A,A\rbrace}=(Q_x,Q_y,Q_z)\subset A$. 
So, $\Omega^2/\overline{\Omega^2}=J(Q)$, the Jacobi ring of $Q$. 

Also, we have the differential ${\bold d}: \Omega^1\to \Omega^2$. By working on $L$, we see that the cokernel of ${\bold d}$ 
consists of the forms $f^p\omega_p$, where $f\in A$, and $\omega_p$ is the image of $\omega$ under the Cartier operator;
if $u,v$ are linear coordinates on $L$ such that $\omega={\bold d}u\wedge {\bold d}v$ then 
$\omega_p=(uv)^{p-1}{\bold d}u\wedge {\bold d}v$. 

\begin{lemma}\label{le1} (i) We have ${\rm Im}{\bold d}+\overline{\Omega^2}=\Omega^2$. 

(ii) There is a natural graded isomorphism $\Omega^1/\overline{\Omega^1}\cong J(Q)$. 

(ii) The kernel $K$ of the map ${\bold d}: A=\Omega^0\to \Omega^1$ is contained in $\overline{\Omega^1}$. 
\end{lemma} 

\begin{proof} (i) Since $J(Q)$ sits in degrees between $0$ and $2h-4$ and $\deg(\omega)=2$, 
the maximal degree of an element of $\Omega^2/\overline{\Omega^2}$ is $2h-2$.
On the other hand, the degree of $f^p\omega_p$ is $\ge 2p$, which is bigger than $2h-2$, proving (i). 

(ii) Consider the embedding $\eta: A\to \Omega^1$ given by $\eta(f)=f(u{\bold d}v-v{\bold d}u)$. 

We claim that $\eta(A)+\overline{\Omega^1}=\Omega^1$. Indeed, we have a map $\tau: \Omega^1\to A$ given by 
$\tau (\alpha {\bold d}u+\beta {\bold d}v)=\alpha u+\beta v$, whose kernel is $\eta(A)$ (by exactness 
of the Koszul complex). The image of this map is the augmentation ideal $I$ in $A$, since 
$x=a^{-1}\tau({\bold d}x)$, $y=b^{-1}\tau({\bold d}y)$, and $z=c^{-1}\tau({\bold d}z)$ (where $c=h$). 
For the same reason $\tau(\overline{\Omega^1})=I$, which implies the statement. 

Next, we claim that $\eta((Q_x,Q_y,Q_z))\subset \overline{\Omega_1}$. 
We will show that $\eta(Q_x)\in \overline{\Omega_1}$; the case of $Q_y$ and $Q_z$ is similar. 
We have 
$$
\eta(Q_x)=\eta(\lbrace{y,z\rbrace})=(y_uz_v-y_vz_u)(u{\bold d}v-v{\bold d}u)=
by{\bold d}z-cz{\bold d}y,
$$
as desired. 

This shows that $\tau$ gives rise to a surjective linear map 
$J(Q)\to \Omega^1/\overline{\Omega^1}$. On the other hand, by (i), we 
have a surjective linear map ${\bold d}: \Omega^1/\overline{\Omega^1}\to \Omega^2/\overline{\Omega^2}=J(Q)$.
This means that both maps are isomorphisms, which proves (ii).  

(iii) By working on $L$, we see that the kernel $K$ is spanned by elements ${\bold d}f$ for $f\in A$ and 
$\alpha^pu^{p-1}{\bold d}u+\beta^pv^{p-1}{\bold d}v$ for $\alpha,\beta\in \Bbb F[L]$, so that $\alpha{\bold d}u+\beta{\bold d}v$ 
is $\Gamma$-invariant (this is just the image of $\alpha{\bold d}u+\beta{\bold d}v$ 
under the Cartier operator). By definition, ${\bold d}f\in \overline{\Omega^1}$.
To see that the second element is also in $\overline{\Omega^1}$, note that it has degree at least $2p$, 
while by (ii), the degree of any homogeneous element of $\Omega^1$ that is not in 
$\overline{\Omega^1}$ is at most $2h-2$, which is smaller. This proves (iii).       
\end{proof}

Lemma \ref{le1} implies that the graded space $\Omega^2$ has a
graded subspace $\Omega^1/K$, which in turn has a graded subspace  
is $\overline{\Omega^1}/K$. This implies that the Hilbert series 
of $\Omega^2/{\bold d}\overline{\Omega^1}$ is 
the sum of the Hilbert series of the corresponding quotients,
i.e., $h_{J(Q)}(t)+t^{2p-2}h_A(t^p)$, which is the right hand side of (\ref{mainform}). 
On the other hand, it easy to see that $\xi^{-1}({\bold d}\overline{\Omega^1})=\lbrace{A,A\rbrace}$, 
so $\tau$ identifies $\Omega^2/{\bold d}\overline{\Omega^1}$ with $HP_0(A)$. 
Theorem \ref{kleini} is proved. 
\end{proof}

\end{document}